\documentclass[12pt]{article}
\usepackage[margin=1in]{geometry}
\usepackage[usenames]{color}
\usepackage[colorlinks=true,
linkcolor=webgreen,
filecolor=webbrown,
citecolor=webgreen]{hyperref}

\definecolor{webgreen}{rgb}{0,.5,0}
\definecolor{webbrown}{rgb}{.6,0,0}

\usepackage{color}

\usepackage{url}
\usepackage{graphicx}
\usepackage{subcaption}

\usepackage[T1]{fontenc}

\usepackage{amsmath}
\usepackage{amssymb}
\usepackage{amsthm}
\usepackage{mathrsfs}

\theoremstyle{plain}
\newtheorem{theorem}{Theorem}

\newtheorem{proposition}[theorem]{Proposition}

\theoremstyle{definition}

\theoremstyle{remark}

\title{The periodic complexity function of the Thue--Morse word,
  the Rudin--Shapiro word, and the period-doubling word}
\author{Narad Rampersad\footnote{
Department of Math/Stats,
University of Winnipeg,
515 Portage Ave.,
Winnipeg, MB, R3B 2E9
Canada; {\tt narad.rampersad@gmail.com}.}}

\begin{document}
\maketitle
\begin{abstract}
  We revisit the periodic complexity function $h_{\bf w}(n)$
  introduced by Mignosi and Restivo.  This function gives the average
  of the first $n$ local periods of a recurrent infinite word
  ${\bf w}$.  We give a different method than that of Mignosi and
  Restivo for computing the asymptotics of the periodic complexity
  function of the Thue--Morse word and show how to apply the method to
  other automatic sequences, like the Rudin--Shapiro word and the
  period-doubling word.
\end{abstract}

\section{Introduction}
Mignosi and Restivo~\cite{MR13} introduced a new complexity measure
for infinite words called the \emph{periodic complexity}.  This
function is defined based on the \emph{local period} at each position
of the infinite word.  Let ${\bf w} = w_0w_1w_2\cdots$ be an infinite
word.  The \emph{periodicity function} $p_{\bf w}(i)$ is defined as
follows.  The value of $p_{\bf w}(i)$ is the length of the shortest
prefix $u$ of $w_iw_{i+1}w_{i+2}\cdots$ such that either \emph{$u$ is
  a suffix of $w_0\cdots w_{i-1}$ or $w_0\cdots w_{i-1}$ is a suffix
  of $u$}, if such a word $u$ exists.  If no such $u$ exists, then
$p_{\bf w}(i)=\infty$.  However, if ${\bf w}$ is recurrent, which will
always be the case in this paper, then $p_{\bf w}(i) < \infty$ for all
$i$.

Since the values of $p_{\bf w}(i)$ can fluctuate wildly, it is not
that suitable as a complexity function.  Mignosi and Restivo therefore
defined the \emph{periodic complexity function} $h_{\bf w}(i)$ as the
\emph{average of the periodicity function}; that is, if
$$P_{\bf w}(i) = \sum_{j=0}^{i-1}p_{\bf w}(j)$$
is the \emph{summatory function} of $p_{\bf w}(i)$, then $h_{\bf
  w}(i) = (1/i)P_{\bf w}(i)$ for $i \geq 1$.

Mignosi and Restivo studied the periodicity function and the
periodicity complexity function for both the \emph{Thue--Morse word}
$${\bf t} = 0110100110010110\cdots$$
and the \emph{Fibonacci word}
$${\bf f} = 0100101001001010\cdots.$$
They proved that $h_{\bf t}(n) = \Theta(n)$ and
$h_{\bf f}(n) = \Theta(\log n)$.  Schaeffer~\cite{Sch13} studied the
periodicity function of Sturmian words using the Ostrowski
representation of natural numbers.

In this paper we study $p_{\bf t}(i)$ and $h_{\bf t}(i)$ with the aid
of the computer program Walnut~\cite{Walnut}.  We get a more precise
description of these functions than the ones given in~\cite{MR13} and
we show how to apply these techniques to other automatic sequences,
such as the Rudin--Shapiro sequence.

\section{Periodic complexity of the Thue--Morse word}
The Thue--Morse word ${\bf t} = t_0t_1t_2\cdots$ is defined by
\[ t_i = \begin{cases} 0 &\text{ if the number of $1$'s in the binary
      representation of $i$ is even,}\\ 1 &\text{ otherwise}. \end{cases} \]
Table~\ref{tab:per_fn_t} shows some initial values of $p_{\bf t}(i)$.
\begin{center}
  \begin{tabular}{|*{17}{c|}}
    \hline
    $i$ & 0 & 1 & 2 & 3 & 4 & 5 & 6 & 7 & 8 & 9 & 10 & 11 & 12 & 13 & 14
    & 15 \\
    \hline
    $p_{\bf t}(i)$ & 1 & 3 & 1 & 6 & 2 & 12 & 1 & 12 & 1 & 24 & 1 & 24
                                                          & 2 & 24 & 1
    & 24 \\
    \hline
  \end{tabular}
  \captionof{table}{Initial values of $p_{\bf t}(i)$\label{tab:per_fn_t}}
\end{center}

We can get a automaton that computes the binary representation of
$p_{\bf t}(i)$ with the following Walnut commands
(see~\cite[Section~10.8.12]{Sha21}):
\begin{verbatim}
def tmEq "?msd_2 Ak (k<n) => T[i+k]=T[j+k]":

def tmRepWd "?msd_2 (i>=n & $tmEq(i,i-n,n)) | (n>i & $tmEq(0,n,i))":

def tmLocPer "?msd_2 (n>0) & $tmRepWd(i,n) & Am (m>0 & m<n) =>
    ~$tmRepWd(i,m)":
\end{verbatim}
This produces the automaton in Figure~\ref{fig:per_fn_t}.
By examining this automaton, one obtains the following result,
which is a more precise version of~\cite[Proposition~3.18]{MR13}.
\begin{figure}[htb]
  \centering
  \includegraphics[scale=0.8]{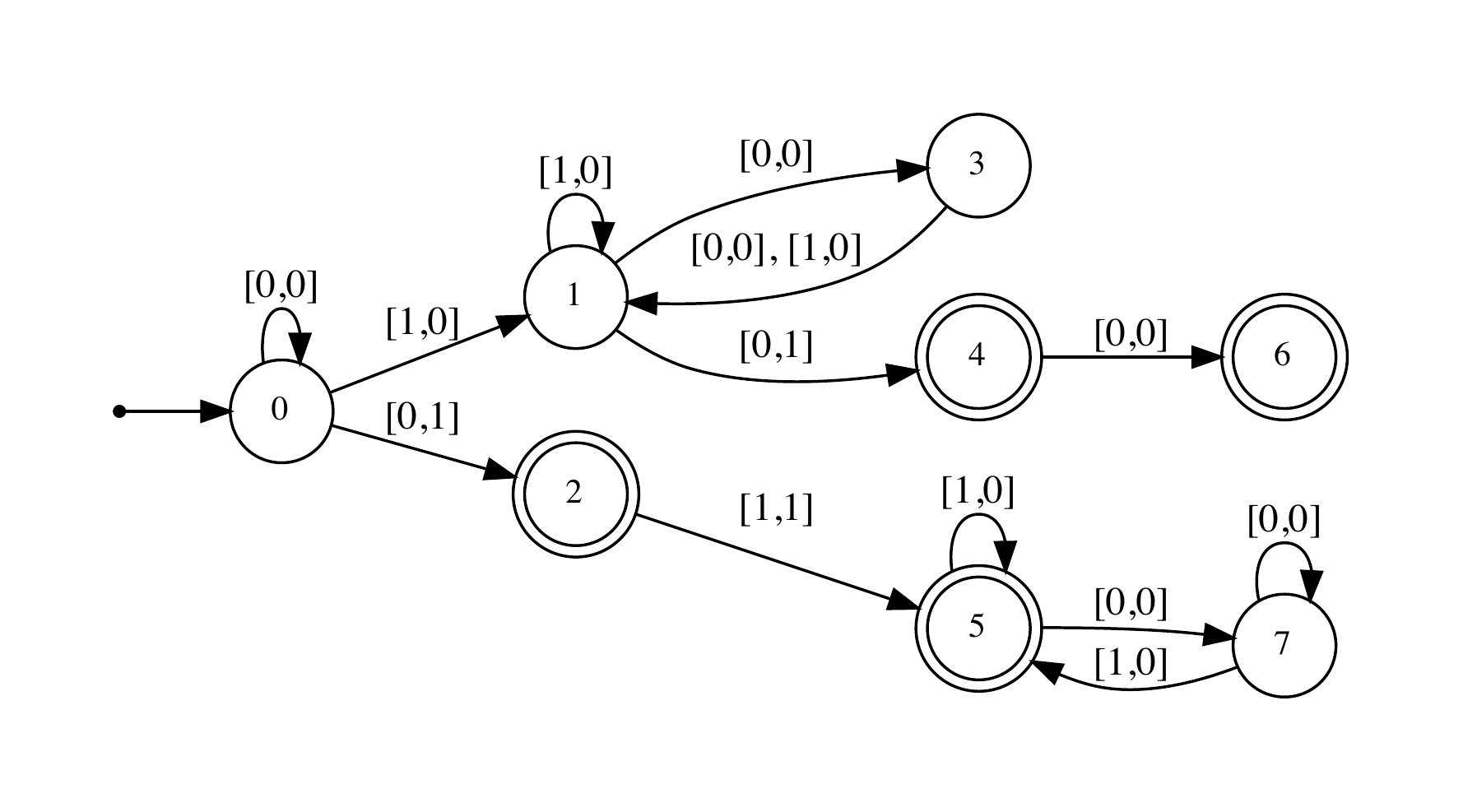}
  \caption{Automaton for the pair $(i,p_{\bf t}(i))$}\label{fig:per_fn_t}
\end{figure}

\begin{proposition}\label{prop:per_fn}
  We have
  \begin{itemize}
  \item $p_{\bf t}(i) \in \{1,2\}$ if $i$ is even; and,
  \item $p_{\bf t}(i) = 3\cdot 2^t$ if $i$ is odd and $2^t \leq
    i < 2^{t+1}$.
  \end{itemize}
\end{proposition}

We can then bound the summatory function of $p_{\bf t}(i)$.

\begin{proposition}\label{prop:sum_fn}
  For $n \geq 1$, we have
  \[ \frac38 (n-1)^2 + \frac{n}{2} \leq P_{\bf t}(n) \leq \frac34 n^2+n+1.\] 
\end{proposition}

\begin{proof}
  We split the sum $P_{\bf t}(n) = \sum_{i=0}^{n-1}p(i)$ into even and
  odd indexed terms.  By Proposition~\ref{prop:per_fn}, we have
  \[ \frac{n}{2} \leq \sum_{\substack{i=0\\i\text{
          even}}}^{n-1}p(i) \leq n+1.\]
  Again, by Proposition~\ref{prop:per_fn}, we have
  \[  \sum_{\substack{i=0\\i\text{
          odd}}}^{n-1} p(i) \leq \sum_{\substack{i=0\\i\text{
          odd}}}^{n-1} 3i \leq 3(n/2)^2 = \frac34 n^2.\]
  and
  \[  \sum_{\substack{i=0\\i\text{
          odd}}}^{n-1} p(i) \geq \sum_{\substack{i=0\\i\text{
          odd}}}^{n-1} 3i/2 \geq (3/2)[(n-1)/2]^2 = \frac38 (n-1)^2.\]
  Hence,
  \[ \frac38 (n-1)^2 + \frac{n}{2} \leq P_{\bf t}(n) \leq \frac34 n^2+n+1.\] 
\end{proof}

This gives the following bounds on the periodic complexity of
${\bf t}$, which are an improvement on the inequalities from the
proof of~\cite[Proposition~3.19]{MR13}.

\begin{theorem}
  For $n \geq 1$, we have
  \[ 3n/8 - 1/4 \leq h_{\bf t}(n) \leq 3n/4+2.\]
  In particular, we have $h_{\bf t}(n) = \Theta(n)$.
\end{theorem}

In this case, we were fortunate that the automaton in
Figure~\ref{fig:per_fn_t} was rather simple.  For more complicated
sequences, this may not be the case, so next we explore other methods
for analyzing the asymptotics of $P_{\bf t}(i)$.  To apply these methods,
we first need a \emph{linear representation} for $p_{\bf t}(i)$.  That
is, we need a integer row vector $v$, an integer column vector $w$,
and a pair of integer matrices $M_0$ and $M_1$, such that
$$
p_{\bf t}(i) = vM_{i_\ell-1}M_{i_{\ell-2}}\cdots M_{i_0}w,
$$
where $i_{\ell-1} i_{\ell-2}\cdots i_0$ is the binary representation of
$i$.  Walnut can produce a linear representation for $p_{\bf t}(i)$
with the command
\begin{verbatim}
eval tmLocPer_enum i "?msd_2 En $tmLocPer(i,n) & m<n & ~$tmLocPer(i,m)":
\end{verbatim}
The output of this command is a Maple worksheet containing the
following values for $v$, $w$, $M_0$ and $M_1$.
\[v = [1,0,1,0,0,0],\]
\[ M_0 =
  \begin{bmatrix}
    1&0&1&0&0&0\\
    0&0&0&0&1&1\\
    0&0&0&0&0&0\\
    0&0&0&0&0&2\\
    0&1&0&1&0&0\\
    0&0&0&0&0&2
  \end{bmatrix},
  \quad\quad
  M_1 =
  \begin{bmatrix}
    0&1&0&1&0&0\\
    0&1&0&1&0&0\\
    0&0&0&1&0&0\\
    0&0&0&2&0&0\\
    0&1&0&1&0&0\\
    0&0&0&2&0&0
  \end{bmatrix},
\]
\[w = [1,1,0,1,1,0]^T.\]

Sequence defined by such linear representations are called
\emph{$2$-regular sequences} (in general, \emph{$q$-regular
  sequences}).  Dumas~\cite{Dum13} obtained a description of the
asymptotics of the summatory function of $q$-regular sequences.
Heuberger and Krenn~\cite{HK20} have also recently given a similar
description of these asymptotics.

To make use of these results, we need a number of
definitions (see~\cite[Section~3.2]{HK20}).  Let
$X(N) = \sum_{n=0}^{N-1} x(n)$ be the summatory function of a sequence
$x(n )$ for which we have a linear representation consisting of a row
vector $v \in \mathbb{C}^d$, a column vector $w \in \mathbb{C}^d$, and
$q$ matrices $M_0, \ldots, M_{q-1} \in \mathbb{C}^{d \times d}$.
That is,
\begin{equation}\label{eq:linrep}
  x(n) = vM_{n_{\ell-1}}M_{n_{\ell-2}}\cdots M_{n_0}w,
\end{equation}
where $n_{\ell-1} n_{\ell-2}\cdots n_0$ is the base-$q$ representation
of $n$.  Let $\|\cdot\|$ denote any norm on $\mathbb{C}^d$, as well as
its induced matrix norm.  Define $M := M_0+M_1+\ldots+M_{q-1}$.
Choose $R>0$ such that
$\|M_{r_1}M_{r_2}\cdots M_{r_\ell}\| = O(R^\ell)$ holds for all
$\ell \geq 0$ and all $r_1,\ldots,r_{\ell} \in \{0,\ldots,q-1\}$.
That is, the number $R$ is an upper bound for the \emph{joint spectral
  radius} of $M_0, \ldots, M_{q-1}$.  Let $\sigma(M)$ denote the set
of eigenvalues of $M$.  For $\lambda \in \mathbb{C}$, if
$\lambda \in \sigma(M)$, let $m(\lambda)$ denote the size of the
largest Jordan block of $M$ associated with $\lambda$, and let
$m(\lambda)=0$ otherwise.  The notation $\{z\}$ denotes the fractional
part of a real number $z$.  The following result is
essentially~\cite[Theorem~1]{Dum13} as presented in the first part
of~\cite[Theorem~A]{HK20}.

\begin{theorem}\label{thm:all-in-one}
With the above definitions, we have
\begin{align*}
  X(N) = &\sum_{\substack{\lambda\in\sigma(M) \\|\lambda|>R}} N^{\log_q
           \lambda} \sum_{0 \leq k <m(\lambda)} \frac{(\log
           N)^k}{k!}\Phi_{\lambda k}(\lbrace \log_q N \rbrace)\\
         &+ O\left(N^{\log_q R}(\log N)^{\max\{m(\lambda) :
           |\lambda|=R\}}\right),           
\end{align*}
where the $\Phi_{\lambda k}$ are certain $1$-periodic continuous
functions.  The big O ``error term'' can be omitted if there are no
eigenvalues $\lambda \in \sigma(M)$ with $|\lambda|\leq R$.
\end{theorem}

Note that we have defined the linear representation of $x(n)$ in terms
of the most-significant-digit first representation of $n$.  It can
also be defined using the least-significant-digit first representation
of $n$ (as it is in~\cite{HK20}).  One can easily convert from one
representation to the other by taking the transpose of $v$,
$M_0,\ldots,M_{q-1}$, $w$, and the transpose of Eq.~\ref{eq:linrep}.
Since the eigenvalues of a matrix and its transpose are the same, we
can still apply Theorem~\ref{thm:all-in-one}, regardless of the choice
of representation.

If we return to the linear representation of $p_{\bf t}(i)$ that we
computed earlier, we have
\[
  M = M_0+M_1 =
  \begin{bmatrix}
    1& 1& 1& 1& 0& 0\\
    0& 1& 0& 1& 1& 1\\
    0& 0& 0& 1& 0& 0\\
    0& 0& 0& 2& 0& 2\\
    0& 2& 0& 2& 0& 0\\
    0& 0& 0& 2& 0& 2
  \end{bmatrix}.
\]
The set of eigenvalues of $M$ is $\sigma(M) = \{4,2,1,0,-1\}$, where
each eigenvalue has multiplicity $1$, except the eigenvalue $0$, which
has multiplicity $2$.  To compute $R$ it is convenient for us to
choose the $\|\cdot\|_\infty$ norm on $\mathbb{C}^6$ (i.e., the
maximum norm), which induces the maximum row sum norm on
$\mathbb{C}^{6\times 6}$.  Since the maximum row sum of $M_0$ and
$M_1$ is $2$, we can take $R=2$.  This is enough information to apply
Theorem~\ref{thm:all-in-one} to $P_{\bf t}(n)$, which gives the
following result.

\begin{theorem}\label{thm:t_Pandh}
  We have
  \[ P_{\bf t}(n) = n^2\Phi_{40}(\lbrace \log_2 n \rbrace) +
    O(n\log n), \]
  and
  \[ h_{\bf t}(n) = n\Phi_{40}(\lbrace \log_2 n \rbrace) +
    O(\log n), \]
  for some $1$-periodic continuous function $\Phi_{40}$.
\end{theorem}

\section{Periodic complexity of the Rudin--Shapiro sequence}
We can determine the asymptotic growth of $h_{\bf x}(n)$ for other
automatic sequences ${\bf x}$ by first using Walnut to compute a
linear representation for $p_{\bf x}(i)$, and then applying
Theorem~\ref{thm:all-in-one}.  Let
\[{\bf rs} = r_0r_1r_2\cdots = 0001001000011101\cdots\]
be the \emph{Rudin--Shapiro sequence}, defined by
\[
  r_i = \begin{cases} 0 &\text{ if the number of $11$'s in the binary
      representation of $i$ is even,}\\ 1 &\text{ otherwise}. \end{cases} \]
If we use Walnut to compute a linear representation for $p_{\bf
  rs}(i)$,
the matrices $M_0$ and $M_1$ that we get are $31\times 31$, so we do
not show them here.  They each have maximum row sum $2$, so again we
can take $R=2$.  The set of eigenvalues of the matrix $M := M_0+M_1$
is $\sigma(M)=\{4,2,1,0,-1,-2\}$.  From the Jordan form of $M$, we
find $m(4)=1$ and $m(2)=m(-2)=2$.  Applying
Theorem~\ref{thm:all-in-one} thus gives the following result.

\begin{theorem}\label{thm:rs_Pandh}
  We have
  \[ P_{\bf rs}(n) = n^2\Phi_{40}(\lbrace \log_2 n \rbrace) +
    O(n\log^2 n), \]
  and
  \[ h_{\bf rs}(n) = n\Phi_{40}(\lbrace \log_2 n \rbrace) +
    O(\log^2 n), \]
  for some $1$-periodic continuous function $\Phi_{40}$.
\end{theorem}

(The function $\Phi_{40}$ in Theorem~\ref{thm:rs_Pandh} is, of course,
different from the one in Theorem~\ref{thm:t_Pandh}.)

\section{Periodic complexity of the period-doubling sequence}

  Next we determine the asymptotic behaviour of $P_{\bf pd}(n)$ and $h_{\bf
    pd}(n)$, where ${\bf pd}$ is the \emph{period-doubling word},
  i.e., the fixed point of the morphism $0 \to 01$, $1 \to 00$.
Table~\ref{tab:per_fn_pd} shows some initial values of $p_{\bf pd}(i)$.
\begin{center}
  \begin{tabular}{|*{17}{c|}}
    \hline
    $i$ & 0 & 1 & 2 & 3 & 4 & 5 & 6 & 7 & 8 & 9 & 10 & 11 & 12 & 13 & 14
    & 15 \\
    \hline
    $p_{\bf pd}(i)$ & 1& 2& 4& 1& 1& 8& 2& 2& 2& 2& 16& 1& 1& 4& 4& 1\\
    \hline
  \end{tabular}
  \captionof{table}{Initial values of $p_{\bf pd}(i)$\label{tab:per_fn_pd}}
\end{center}
  Our goal is to show that $h_{\bf pd}(n) = \Theta(\log n)$ (i.e., its periodic
complexity is rather more like that of the Fibonacci word than the
Thue--Morse word).

We begin by using Walnut to compute the following linear representation for
$p_{\bf pd}(i)$:
\[v = [1,0,0,0,0,0],\]
\[ M_0 =
  \begin{bmatrix}
    1&0&0&0&0&0\\
    0&0&0&1&0&1\\
    0&0&0&0&0&2\\
    0&0&0&0&1&0\\
    0&0&0&1&0&1\\
    0&0&0&0&0&0
  \end{bmatrix},
  \quad\quad
  M_1 =
  \begin{bmatrix}
    0&1&1&0&0&0\\
    0&0&0&0&1&0\\
    0&0&0&0&0&0\\
    0&1&1&0&0&0\\
    0&1&1&0&0&0\\
    0&0&2&0&0&0
\end{bmatrix},
\]
\[w = [1,1,1,1,1,1]^T.\]

Now, if we try to apply Theorem~\ref{thm:all-in-one} to $P_{\bf pd}(n)$, we
run into the following problem.  The maximum row sum of $M_0$ and
$M_1$ is $2$, so we could take $R=2$, but the largest eigenvalue of
$M := M_0+M_1$ is also $2$ (with $m(2)=2$), which means that in this
case the ``error term'' in Theorem~\ref{thm:all-in-one} dominates, and
we don't obtain the desired asymptotics.  However, using other
methods, we can obtain the following bounds.

\begin{theorem}\label{thm:pd_Pandh}
  For $n \geq 1$, we have
  \[ (1/3\log_2 n -1/18)n+4/9 \leq P_{\bf pd}(n) \leq (4/3\log_2 n +22/9)n+5/9\]
  and
  \[ 1/3\log_2 n -1/18 \leq h_{\bf pd}(n) \leq 4/3\log_2 n +3.\]
\end{theorem}

\begin{proof}
  For $\ell \geq 0$, we have
  \begin{align*}
    P_{\bf pd}(2^\ell) &= \sum_{i < 2^\ell} p(i) \\
    &= \sum_{i_0,\ldots,i_{\ell-1} \in \{0,1\}} vM_{i_{\ell-1}}\cdots
      M_{i_0}w \\
              &= v(M_0+M_1)^\ell w \\
              &= vM^\ell w.
  \end{align*}

To obtain a formula for $vM^\ell w$, we first compute the
minimal polynomial of $M$:
\[ m_M(x) = (x-2)^2(x+2)(x-1)(x+1). \]
It follows then that
\begin{equation}\label{eq:pd_linrec}
vM^\ell w = (A+B\ell)2^\ell + C(-2)^\ell + D +
E(-1)^\ell,
\end{equation}
for some constants $A, \ldots, E$.  To compute these constants, we
compute $vM^\ell w$ (i.e., $P_{\bf pd}(2^\ell)$) for
$\ell = 0,\ldots,4$, which gives the values $1, 3, 8, 21, 52$.
We then substitute these values into \eqref{eq:pd_linrec} to obtain a
system of linear equations in the variables $A, \ldots, E$.  When we
solve this system of linear equations we get
$$A=5/9, B=2/3, C=0, D=1/2, E=-1/18.$$  Thus, we have
\[
  P_{\bf pd}(2^\ell) = (5/9+(2/3)\ell)2^\ell + 1/2  - 1/18(-1)^\ell,
\]
and so
\[
  (5/9+(2/3)\ell)2^\ell + 4/9 \leq P_{\bf pd}(2^\ell) \leq (5/9+(2/3)\ell)2^\ell + 5/9.
\]

Now write $2^\ell \leq n < 2^{\ell+1}$, so that $\ell \leq \log_2 n <
\ell+1$.  Then
\begin{align*}
  (5/9+(2/3)\ell)2^\ell + 4/9
  &\leq P_{\bf pd}(n)
  \leq (5/9+2/3(\ell+1))2^{\ell+1} + 5/9 \\
  (5/9+2/3(\log_2 n-1))(n/2) + 4/9
  &\leq P_{\bf pd}(n)
  \leq (5/9+2/3(\log_2 n+1))(2n) + 5/9 \\
  (1/3\log_2 n-1/18)n + 4/9
  &\leq P_{\bf pd}(n)
  \leq (4/3\log_2 n+22/9)n + 5/9,
\end{align*}
and
\[
  1/3\log_2 n-1/18 \leq h_{\bf pd}(n) \leq 4/3\log_2 n +3.
\]
\end{proof}

\section*{Acknowledgments}
We thank Jeffrey Shallit for suggesting the approach used in the proof
of Theorem~\ref{thm:pd_Pandh}.

\end{document}